\documentclass[12pt,letterpaper]{article}

\usepackage[utf8]{inputenc}

\usepackage{amssymb,amsmath,amsthm,amsfonts}
\usepackage[pdftex,colorlinks=true,linkcolor=blue,citecolor=blue,urlcolor=red,unicode=true,hyperfootnotes=false,bookmarksnumbered]{hyperref}
\usepackage[margin=1truein]{geometry}

\usepackage{enumerate}

\theoremstyle{plain}
\newtheorem{theorem}{Theorem}[section]
\newtheorem{lemma}[theorem]{Lemma}
\newtheorem{claim}[theorem]{Claim}

\newtheorem{corollary}[theorem]{Corollary}

\theoremstyle{definition} 

\newtheorem{remark}[theorem]{Remark}

\makeatletter
\def\@gifnextchar#1#2#3{\let\@tempe#1\def\@tempa{#2}\def\@tempb{#3}%
  \futurelet\@tempc\@gifnch}

\def\@gifnch{\ifx\@tempc\@sptoken\let\@tempd\@tempb%
  \else\ifx\@tempc\@tempe\let\@tempd\@tempa\else\let\@tempd\@tempb\fi\fi\@tempd}

\def\SK@set#1{\left\{#1\right\}}
\def\SK@@set#1#2{\{#1\,:\,
    \begin{array}{@{}l@{}}#2\end{array}
\}}
\def\SK@mset#1{\left\{\!\!\left\{#1\right\}\!\!\right\}}
\def\SK@@mset#1#2{\{\!\!\{#1\,:\,
    \begin{array}{@{}l@{}}#2\end{array}
\}\!\!\}}
\def\BIG@set#1{\Big\{#1\Big\}}
\def\BIG@@set#1#2{\Big\{#1\:\Big|\:
    \begin{array}{@{}l@{}}#2\end{array}
\Big\}}
\newcommand{\Set}[1]{\@gifnextchar\bgroup{\SK@@set{#1}}{\SK@set{#1}}}
\newcommand{\Mset}[1]{\@gifnextchar\bgroup{\SK@@mset{#1}}{\SK@mset{#1}}}
\newcommand{\Bigset}[1]{\@gifnextchar\bgroup{\BIG@@set{#1}}{\BIG@set{#1}}}
\makeatother

\title{Weak saturation in graphs: a combinatorial approach}

\author{Nikolai Terekhov\thanks{Department of Discrete Mathematics, Moscow Institute of Physics and Technology, Dolgoprudny, Russia; nikolayterek@gmail.com
}, Maksim Zhukovskii\thanks{Department of Computer Science, University of Sheffield, Sheffield S1 4DP, UK; zhukmax@gmail.com}
}

\date{}

\begin{document}

\maketitle

\begin{abstract}
The weak saturation number $\mathrm{wsat}(n,F)$ is the minimum number of edges in a graph on $n$ vertices such that all the missing edges can be activated sequentially so that each new edge creates a copy of $F$. A usual approach to prove a lower bound for the weak saturation number is algebraic	: if it is possible to embed edges of $K_n$ in a vector space in a certain way (depending on $F$), then the dimension of the subspace spanned by the images of the edges of $K_n$ is a lower bound for the weak saturation number. In this paper, we present a new combinatorial approach to prove lower bounds for weak saturation numbers that allows to establish worst-case tight (up to constant additive terms) general lower bounds as well as to get exact values of the weak saturation numbers for certain graph families. It is known (Alon, 1985) that, for every $F$, there exists $c_F$ such that $\mathrm{wsat}(n,F)=c_Fn(1+o(1))$. Our lower bounds imply that all values in the interval $\left[\frac{\delta}{2}-\frac{1}{\delta+1},\delta-1\right]$ with step size $\frac{1}{\delta+1}$ are achievable by $c_F$ (while any value outside this interval is not achievable).
\end{abstract}

\section{Introduction}
\label{sc:intro}

Given a graph $F$, an {\it $F$-bootstrap percolation process} is a sequence of graphs $H_0\subset H_1\subset\cdots\subset H_m$ such that, for   $i=1,\ldots,m$, $H_i$ is obtained from $H_{i-1}$ by adding an edge that belongs to a copy of $F$ in $H_i$. The $F$-bootstrap percolation process was introduced by Bollob\'{a}s \cite{bol} and can be seen as a special case of cellular automata. This notion is also related to $r$-neighborhood bootstrap percolation model having applications in physics; see, for example, \cite{Adler, Fontes,Morris}. Given $n\in\mathbb{N}$ and a graph $F$,  we call a graph $H$ on $[n]:=\{1,\ldots,n\}$  {\it weakly $F$-saturated}, if there exists an $F$-bootstrap percolation process $H=H_0\subset H_1\subset\cdots\subset H_m=K_n$. The  minimum number of edges in a weakly $F$-saturated graph is called the {\it weak $F$-saturation number} and is denoted by $\mathrm{wsat}(n, F)$. We will also denote by $\mathrm{wSAT}(n,F)$ the set of all weakly $F$-saturated graphs and by $\underline{\mathrm{wSAT}}(n,F)$ the set of those of them that have exactly $\mathrm{wsat}(n,F)$ edges. 

Throughout the paper we consider graphs $F$ without isolated vertices since, obviously, for $n\geq|V(F)|$, the value of $\mathrm{wsat}(n,F)$ coincides with the weak saturation number of the graph that obtained from $F$ be removing the isolated vertices. Everywhere below, we denote by $v$, $\ell$ and $\delta$ the number of vertices in $F$, the number of edges in $F$ and the minimum degree of $F$ respectively.

Note that a graph obtained from an $H\in\mathrm{wSAT}(v,F)$ by drawing from every vertex of $[n]\setminus [v]$ exactly $\delta-1$ edges to the vertices of $H$ belongs to $\mathrm{wSAT}(n,F)$. This observation immediately gives the upper bound
\begin{equation}
 \mathrm{wsat}(n,F)\leq \mathrm{wsat}(v,F)+(n-v)(\delta-1) \leq {v\choose 2}-1+(n-v)(\delta-1).
\label{wsat:upper_bound_general}
\end{equation}
This bound is sharp since $\mathrm{wsat}(n,K_v)={v\choose 2}-1+(n-v)(v-2)$~\cite{L77} (alternative proofs of this famous result have been obtained in \cite{Alon85,Frankl82,Kalai85,Kalai84}). Another example that shows optimality of the first inequality in (\ref{wsat:upper_bound_general}) is a star graph $F=K_{1,v-1}$ (see, e.g.,~\cite{Kal_Zhuk}).

The best known 	general lower bound for $\mathrm{wsat}(n,F)$ is due to Faudree, Gould and Jacobson~\cite{Faudree}. They showed that for graphs $F$ with minimum degree $\delta$
\begin{equation}
 \mathrm{wsat}(n,F)\geq\left(\frac{\delta}{2}-\frac{1}{\delta+1}\right)n
\label{wsat:lower_bound_general}
\end{equation}
for $n$ sufficiently large. This bound is true (even for all $n\geq v$). However, the argument in the form presented in the paper seems to be false. It is very short, and so, for convenience, we duplicate it in Appendix as well as explain the issue in the proof. 

In this paper, we prove (\ref{wsat:lower_bound_general}) by giving a new general lower bound expressed in the terms of an invariant of the graph $F$ that equals to the vector $(e_i,i\in[v])$, where $e_i$ is 1 less than the minimum number of edges adjacent to a vertex from a fixed $i$-set in $F$. The bound is presented in Section~\ref{sc:main}. Below, we state its more explicit corollary for certain graph families. Note that in the worst case our bound is slightly stronger than (\ref{wsat:lower_bound_general}). Also, it implies even better (and almost sharp) bounds for connected graphs $H$ with $\delta>1$. Note that, from the upper bound (\ref{wsat:upper_bound_general}), it immediately follows that $\mathrm{wsat}(n,F)=O(1)$ if $\delta=1$ (also, the lower bound (\ref{wsat:lower_bound_general}) becomes trivial). Therefore, we restrict ourselves with $\delta>1$.

Let us state the new lower bound. Recall that a graph is called {\it $k$-edge-connected}, if it remains connected whenever at most $k-1$ edges are removed.\\

\begin{theorem}
Let $\delta>1$. Then, for all $n\geq v$, 
$$
 \mathrm{wsat}(n,F)\geq\left(\frac{\delta}{2}-\frac{1}{\delta+1}\right)(n-v)+\ell-1.
$$
If $\delta$ is odd and $F$ is connected, then, for all $n\geq v$, 
$$
 \mathrm{wsat}(n,F)\geq\left(\frac{\delta}{2}-\frac{1}{2(\delta+2)}\right)(n-v)+\ell-1.
$$
If $\delta$ is even and $F$ is connected, or $\delta$ is arbitrary and $F$ is 2-edge-connected, then, for all $n\geq v$, 
$$
 \mathrm{wsat}(n,F)\geq\frac{\delta}{2}(n-v)+\ell-1.
$$
\label{th:new_lower_bound}
\end{theorem}

The proof of this theorem is given in Section~\ref{sc_sub:gamma}. Then, in Section~\ref{sc:optimality}, we show that all our bounds are sharp up to a constant additive term. Note that, since $\ell\geq\frac{1}{2}v\delta$, the first bound implies $\mathrm{wsat}(n,F)\geq\left(\frac{\delta}{2}-\frac{1}{\delta+1}\right)n+\frac{v}{\delta+1}-1$ which is better than (\ref{wsat:lower_bound_general}). Indeed, $\frac{v}{\delta+1}-1\geq 0$, and the equality holds only when $F=K_v$ for which the upper bound is the answer.\\

It was observed by Alon~\cite{Alon85} that, for every graph $F$,  $\mathrm{wsat}(n,F)=c_F n(1+o(1))$ for some constant $c_F\geq 0$. Though the possible values of $c_F:=\lim_{n\to\infty}\frac{\mathrm{wsat}(n,F)}{n}$ are unknown, we have that, for all $F$ with $\delta\geq 2$, $\frac{\delta}{2}-\frac{1}{\delta+1}\leq c_F\leq\delta-1$, and both bounds are achievable. If $\delta=1$, then the only possible value of $c_F$ is 0. It is natural to ask, how are the values of $c_F$ distributed in this interval, if $\delta\geq 2$ is fixed? We have proved that they are not concentrated around the endpoints.\\

\begin{theorem}
For every integer $\delta\geq 2$, every $k\in\{0,1,\ldots,(\delta/2-1)(\delta+1)\}$, and every integer $N$ there exists a connected graph $F$ with the minimum degree $\delta$ and $|V(F)|\geq N$ such that
$$
 \mathrm{wsat}(n,F)=\left(\frac{\delta}{2}+\frac{k}{\delta+1}\right)n+O(1).
$$
\label{th:non_concentration}
\end{theorem}

In other words, all values between $\frac{\delta}{2}-\frac{1}{\delta+1}$ and $\delta-1$ with step size $\frac{1}{\delta+1}$ are achievable by $c_F$. Theorem~\ref{th:non_concentration} is proven in Section~\ref{sc:non_concentration}.

In~\cite{Tuza} it was conjectured by Tuza that, for every graph $F$, $\mathrm{wsat}(n,F)=c_F n +O(1)$. The conjecture clearly holds for graphs with minimum degree $\delta=1$. The last part of Theorem~\ref{th:new_lower_bound} together with the upper bound (\ref{wsat:upper_bound_general}) imply that the conjecture is also true for all connected graphs with $\delta=2$. In the proof of Theorem~\ref{th:new_lower_bound} we introduce the parameter $\gamma$ that equals to the minimum value of $\gamma(S):=(|E(F)|-|E(F\setminus S)|-1)/|S|$ 
 over all proper $S\subset V(F)$. Note that $\gamma$ is exactly the constant in front of $n$ in our lower bounds. It is natural to ask, if, for any $F$, there exists an $S\subset V(F)$ such that $c_F=\gamma(S)$. This turns out to be true for all $F$ that we are familiar of. In particular, this is true not only for graphs with $\delta=1$, for connected graphs with $\delta=2$, and for cliques, but also for a (not necessarily disjoint) union of arbitrary number of cliques of sizes at least $\delta+1$ (see Remark~\ref{rk:union_of_cliques}). Moreover, it is not hard to prove that the conjecture of Tuza is true for all graphs with $c_F=\delta-1$ (in particular, for unions of cliques as above), and that if it is not true for some $F$, then the second-order term should approach $+\infty$ at least for some sequence of $n$. Indeed, this follows directly from the upper bound (\ref{wsat:upper_bound_general}) and the lower bound given in the claim below.
 
\begin{claim}
For every graph $F$, $\mathrm{wsat}(n,F)\geq c_F (n-v) +\ell-1$.
\label{cl:wsat_lower_bound_additive_constant}
\end{claim}

\begin{proof}
The assertion of the claim follows immediately from 
\begin{itemize}
\item the fact that the function $g_F:\mathbb{Z}_{\geq 0}\to\mathbb{Z}_{\geq 0}$ defined as $g_F(i)=\mathrm{wsat}(v+i,F)-(\ell-1)$ is subadditive (see Claim~\ref{cl:wsat-subadditive} in Section~\ref{sc:proof-th:bridges}), and 
\item the obvious bound $g(i)\geq (\lim_{x\to\infty}g(x)/x)i$ that holds for any subadditive function $g:\mathbb{Z}_{\geq 0}\to\mathbb{R}$ (indeed, if the opposite inequality holds for some $i>0$, then $\frac{g(ji)}{ji}\leq\frac{g(i)}{i}<\lim_{x\to\infty}\frac{g(x)}{x}$ --- a contradiction).
\end{itemize}
\end{proof}




In~\cite{Faudree}, the authors also tried to go beyond $F=K_v$ and considered $F_{v,\delta}$ obtained from $K_v$ by removing $(v-1-\delta)$ edges adjacent to the same vertex. They conjectured that the upper bound is tight for these graphs, i.e. $\mathrm{wsat}(n,F_{v,\delta})={v-1\choose 2}+(n-v+1)(\delta-1)$ and prove the conjecture only for $v=5$ and $\delta=3$. We prove the conjecture for all $v$ and $\delta$ in Section~\ref{sc:faudree_conjecture}.  

In Section~\ref{sc:general_tight} (which is, together with the proofs of the main theorems from that section given in Sections~\ref{sc:proof-th:g_2}~and~\ref{sc:proof-th:bridges}, probably the most involved part of the paper), we demonstrate that our method is very powerful for certain graph families, so that it allows to find {\it exact} values of weak saturation numbers. We make a deeper analysis of our techniques and refine our upper bounds from Section~\ref{sc:connected} (that imply tightness of the bounds in Theorem~\ref{th:new_lower_bound}). In particular, these refined bounds imply exact values of the weak saturation numbers for the families of graphs $F$ that are obtained from two cliques by drawing several edges between them. A motivation for considering these graphs is that these are, probably, the most straightforward examples of graphs having weak saturation numbers rather close to the lower bounds. Moreover, we show that the bounds from Section~\ref{sc:general_tight} imply the exact value of the weak saturation number for all connected graphs $F$ that are not 4-edge-connected and satisfy $\mathrm{wsat}(v,F)=\ell-1$. These bounds also imply that $\mathrm{wsat}(n,K_4)=2n-3$ (though a combinatorial proof of this fact was known~\cite{bol}). The formulations of all these results appear in Section~\ref{sc:general_tight} but not in Introduction since we do not want to overload it with massive notations needed for that. 

\section{The lower bound}
\label{sc:main}

As was noted in~\cite{Alon85}, the existence of $\lim_{n\to\infty}\frac{\mathrm{wsat}(n,F)}{n}$ is immediate due to the fact that $\mathrm{wsat}(n,F)+(v-2)^2$ is subadditive. Indeed, it is easy to see that, if $[n]$ is divided into parts $[m]$ and $[n]\setminus[m]$ and two graphs isomorphic to some graphs from $\mathrm{wSAT}(m,F)$ and $\mathrm{wSAT}(n-m,F)$ are constructed on $[m]$ and $[n]\setminus[m]$ respectively, then it remains to draw all edges between fixed $(v-2)$-sets in $[m]$ and $[n]\setminus[m]$ to make the final graph weakly $F$-saturated, implying that $\mathrm{wsat}(n,F)\leq\mathrm{wsat}(m,F)+\mathrm{wsat}(n-m,F)+(v-2)^2$. 

Therefore, it is natural to construct a lower bound in the form $g(n)+O(1)$, where $g(n)$ is a subadditive function. We show that any such subadditive function is suitable unless, for some $i\in\{0,1,\ldots,v\}$, $g(i)$ exceeds the value of $e_i$ mentioned in Section~\ref{sc:intro}. More formally, for $i\in[v]$, set 
$$
e_i:=\min_{S\in{V(F)\choose i}}|E(F)\setminus E(F\setminus S)|-1,
$$
where hereinafter $F\setminus S$ is a subgraph of $F$ induced on $|V(F)|\setminus S$. Note that $e_1=\delta-1$. Set $e_0=0$.

\begin{theorem}
Let $g:\mathbb{Z}_{\geq 0}\to\mathbb{R}$ satisfy the following conditions:
\begin{itemize}
\item for every $i,j\in\mathbb{Z}_{\geq 0}$, $g(i+j)\leq g(i)+g(j)$;
\item for every $i\in\{0,1,\ldots,v-1\}$, $g(i)\leq e_i$.
\end{itemize}
Then, for every integer $n\geq v$,
$$
 \mathrm{wsat}(n,F)\geq g(n-v)+\ell-1.
$$
\label{th:main_lower_subadditive}
\end{theorem}

Since the function $g^*:\mathbb{Z}_{\geq 0}\to\mathbb{R}$, such that, for all $i$, $g^*(i)$ is defined as supremum of $g(i)$ over all $g$ satisfying the conditions of Theorem~\ref{th:main_lower_subadditive}, satisfies these conditions as well, we get that $g^*(n-v)+\ell-1$ is the best possible bound that follows from Theorem~\ref{th:main_lower_subadditive}. Let us show that it is not hard to define $g^*$ explicitly. Set 
$$
g^*(0)=0,\quad g^*(i)=\min_{s\in[i],\,1\leq i_1,\ldots,i_s\leq v-1:\,i_1+\ldots+i_s=i}(e_{i_1}+\ldots+e_{i_s}).
$$
Obviously the conditions in Theorem~\ref{th:main_lower_subadditive} hold for $g^*$. Moreover, if $g$ satisfies these conditions and, for some $i$, $g(i)>g^*(i)$, then find $s\in[i]$ and $i_1,\ldots,i_s$ such that $g^*(i)=e_{i_1}+\ldots+e_{i_s}$. Since $g$ is subadditive, we get 
$$
g(i)\leq g(i_1)+\ldots+g(i_s)\leq e_{i_1}+\ldots+e_{i_s}=g^*(i)
$$ 
--- a contradiction.\\

{\it Proof of Theorem~\ref{th:main_lower_subadditive}.} Let $n\geq v$. For $i\geq v$, set $f(i)=g(i-v)+\ell-1$. Let $H\in\mathrm{wSAT}(n,F)$. Let $\mathcal{O}_H$ be the set of all vectors $(B_1,\ldots,B_k)$ such that
\begin{itemize}
\item for every $\kappa\in[k]$, $B_{\kappa}\subset[n]$, $|B_{\kappa}|\geq v$,
\item for $\kappa_1\neq\kappa_2$, $B_{\kappa_1}\cap B_{\kappa_2}=\varnothing$,
\item for every $\kappa\in[k]$, $|E(H|_{B_{\kappa}})|\geq f(|B_{\kappa}|)$,
\item $|B_1|\geq|B_2|\geq\ldots\geq|B_k|$.
\end{itemize}
Note that $\mathcal{O}_H$ is non-empty. Indeed, consider the first edge added to $H$ in an $F$-bootstrap percolation process. This edge creates a copy $\tilde F$ of $F$. Then, clearly, $(V(\tilde F))\in\mathcal{O}_H$. Indeed $|E(H|_{V(\tilde F)})|\geq |E(\tilde F)|-1=|E(F)|-1$. Also note that if, for every $H$, $([n])\in\mathcal{O}_H$, then we get the statement of Theorem~\ref{th:main_lower_subadditive} immediately.

For $\mathcal{B}_{\alpha}=(B_1^{\alpha},\ldots,B_{k_{\alpha}}^{\alpha})\in\mathcal{O}_H$, $\alpha\in\{1,2\}$, set $\mathcal{B}_1\leq\mathcal{B}_2$, if $(|B_1^1|,\ldots,|B_{k_1}^1|)\leq (|B_1^1|,\ldots,|B_{k_2}^2|)$ in the lexicographical order. Let $\mathcal{B}^*$ be a maximal element of $\mathcal{O}_H$. The following lemma concludes the proof of Theorem~\ref{th:main_lower_subadditive}. $\Box$\\

\begin{lemma}
$\mathcal{B}^*=([n])$.
\label{lm:max_decomposition}
\end{lemma}

{\it Proof of Lemma~\ref{lm:max_decomposition}.} Assume the contrary: let $\mathcal{B}^*=(B_1,\ldots,B_k)\neq([n])$. Take $v\notin B_1$. If $v$ is adjacent to all vertices from $B_1$ in $H$, then, due to the properties of $g$,
\begin{align*}
 |E(H|_{B_1\cup\{v\}})| &= |E(H|_{B_1})|+|B_1|\geq f(|B_1|)+v=g(|B_1|-v)+v+\ell-1\\ 
  & > g(|B_1|-v)+e_1+\ell-1\geq g(|B_1|-v)+g(1)+\ell-1\\
  &\geq g(|B_1|-v+1)+\ell-1=f(|B_1|+1)=f(|B_1\cup\{v\}|).
\end{align*} 
Therefore $(B_1\cup\{v\})\in\mathcal{O}_H$  --- a contradiction with the maximality of $\mathcal{B}^*$. 

Therefore, there exists an edge $\{u,v\}$ such that its endpoints do not belong to the same $B_i$. Let $\{u,v\}$ be the first such edge in an $F$-bootstrap percolation process that starts on $H$. Let this edge, when added, creates a copy $\tilde F$ of $F$. If $\tilde F$ does not meet any of $B_{\kappa}$, $\kappa\in[k]$, then $(B_1,\ldots,B_k,V(\tilde F))\in\mathcal{O}_H$, that contradicts with the maximality of $\mathcal{B}^*$. Let $\mathcal{I}$ be the non-empty set of all $\kappa\in[k]$ such that $B_{\kappa}\cap V(\tilde F)\neq\varnothing$. For every $\kappa\in\mathcal{I}$, set $W_{\kappa}=B_{\kappa}\cap V(\tilde F)$. Let $\tilde B=\bigcup_{\kappa\in\mathcal{I}}B_{\kappa}\cup V(\tilde F)$. Let $\kappa^*=\min\mathcal{I}$. Let us prove that $(B_1,\ldots,B_{\kappa^*-1},\tilde B)\in\mathcal{O}_H$ and derive to a contradiction with the maximality of $\mathcal{B}^*$. Since all the edges of $\tilde F\setminus\{u,v\}$ that do not lie inside any of $B_{\kappa}$ belong to the initial graph $H$, we get that there are 
$$
\left|E\left(H|_{V(\tilde F)}\right)\right|-
\sum_{\kappa\in\mathcal{I}}\left|E\left(H|_{W_{\kappa}}\right)\right|\geq
 \ell-1-\sum_{\kappa\in\mathcal{I}}|E(\tilde F|_{W_{\kappa}})|\geq
 \ell-1-\sum_{\kappa\in\mathcal{I}}(\ell-(e_{v-|W_{\kappa}|}+1))
$$ such edges. Therefore,
\begin{align*} 
|E(H|_{\tilde B})| & =
\sum_{\kappa\in\mathcal{I}}\left|E\left(H|_{B_{\kappa}}\right)\right|+
\left[\left|E\left(H|_{V(\tilde F)}\right)\right|-
\sum_{\kappa\in\mathcal{I}}\left|E\left(H|_{W_{\kappa}}\right)\right|\right]\\
 &\geq \sum_{\kappa\in\mathcal{I}}f(|B_{\kappa}|)+\left[\ell-1-\sum_{\kappa\in\mathcal{I}}(\ell-(e_{v-|W_{\kappa}|}+1))\right]\\
 &\geq\sum_{\kappa\in\mathcal{I}}(f(|B_{\kappa}|)-\ell+1+g(v-|W_{\kappa}|))+\ell-1\\
 & =\sum_{\kappa\in\mathcal{I}}(g(|B_{\kappa}|-v)+g(v-|W_{\kappa}|))+\ell-1\\
 &\geq g\left(\sum_{\kappa\in\mathcal{I}}(|B_{\kappa}|-|W_{\kappa}|)\right)+\ell-1=g(|\tilde B|-v)+\ell-1=f(|\tilde B|). \quad\quad \Box
\end{align*}

\section{Lower bounds for graph families}
\label{sc:connected}

In this section we prove Theorem~\ref{th:new_lower_bound} and, after that, show that the bounds are best possible, up to an additive constant.

\subsection{Proof of Theorem~\ref{th:new_lower_bound}}
\label{sc_sub:gamma}

Recall that $\delta>1.$

We will use Theorem~\ref{th:main_lower_subadditive}. Within the notations of Section~\ref{sc:main}, let $\gamma=\min_{1\leq i\leq v-1}\frac{e_i}{i}$. Set $g(n)=\gamma n$. Clearly, $g$ satisfies the conditions in Theorem~\ref{th:main_lower_subadditive}. Therefore, $\mathrm{wsat}(n,F)\geq \gamma(n-v)+\ell-1$. It remains to apply the claim stated below.\\

\begin{claim} The following lower bounds on $\gamma$ hold.
\begin{enumerate}
\item $\gamma\geq\frac{\delta}{2}-\frac{1}{\delta+1}$.
\item If $\delta$ is even and $F$ is connected, or $\delta$ is arbitrary and $F$ is 2-edge-connected, then $\gamma\geq\frac{\delta}{2}$.
\item If $\delta$ is odd and $F$ is connected, then $\gamma\geq\frac{\delta}{2}-\frac{1}{2(\delta+2)}$.
\end{enumerate}
\end{claim}

\proof Let $i\in[v-1]$. Since $\delta$ is the minimum degree of $F$, for $S\in{V(F)\choose i}$, the summation of degrees of all vertices from $S$ is at least $\delta i$. On the other hand, it is exactly $2e[S]+e[S,V(F)\setminus S]$, where $e[S]=|E(F|_S)|$ is the number of edges inside $S$, and $e[S,V(F)\setminus S]$ is the number of edges between the vertices of $S$ and the vertices of $V(F)\setminus S$. Therefore, 
\begin{equation}
e[S]\geq \left\lceil\frac{\delta i-e[S,V(F)\setminus S]}{2}\right\rceil
\label{eq:count_edges_inside}
\end{equation}
and
\begin{equation}
 |E(F)\setminus E(F\setminus S)|=e[S]+e[S,V(F)\setminus S]\geq \delta i-e[S]\geq\delta i-{i\choose 2}.
\label{eq:count_all_edges_small_k}
\end{equation}
Let $\lambda=I(F\text{ is connected })+I(F\text{ is 2-edge-connected})$. Clearly, $e[S,V(F)\setminus S]\geq\lambda$.  From (\ref{eq:count_edges_inside}), we get
\begin{equation}
 |E(F)\setminus E(F\setminus S)|=e[S]+e[S,V(F)\setminus S]\geq\frac{\delta i+\lambda}{2}I(\delta i-\lambda\text{ is even})+\frac{\delta i+\lambda+1}{2}I(\delta i-\lambda\text{ is odd}).
\label{eq:count_all_edges_large_k}
\end{equation}

Combining (\ref{eq:count_all_edges_small_k}) with (\ref{eq:count_all_edges_large_k}), we get that
$$
 \gamma\geq\min\left\{\min_{1\leq i\leq\delta}\left[\delta-\frac{i-1}{2}-\frac{1}{i}\right],
 \min_{\delta+1\leq i\leq v-1}\frac{\delta i+\lambda-2+I(\delta i-\lambda\text{ is odd})}{2i}\right\}
$$
Therefore, for odd $\delta$, 
$$
 \gamma\geq
 \min\left\{\frac{\delta}{2}+\frac{1}{2}-\frac{1}{\delta},
 \frac{\delta}{2}+\left[-\frac{I(\lambda=0)}{\delta+1}
 -\frac{I(\lambda=1)}{2(\delta+2)}\right]\right\}=
 \frac{\delta}{2}-\frac{I(\lambda=0)}{\delta+1}
 -\frac{I(\lambda=1)}{2(\delta+2)}.
$$
For even $\delta$,
$$
 \gamma\geq
 \min\left\{\frac{\delta}{2}+\frac{1}{2}-\frac{1}{\delta},
 \frac{\delta}{2}-\frac{I(\lambda=0)}{\delta+1}\right\}=
 \frac{\delta}{2}-\frac{I(\lambda=0)}{\delta+1}.\quad\Box
$$

\subsection{Optimality}
\label{sc:optimality}

Here we show that the bounds in Theorem~\ref{th:new_lower_bound} are optimal up to an additive constant term. Everywhere below, as usual, $\delta\geq 2$ is the minimum degree of $F$; $v$ and $\ell$ are the number of vertices and the number of edges respectively.\\

\begin{theorem}
There exists $C>0$ such that, for every $x\in\mathbb{Z}_{\geq 0}$,
\begin{enumerate}
\item there exists a graph $F$ with at least $x$ vertices such that 
$$
\mathrm{wsat}(n,F)\leq\left(\frac{\delta}{2}-\frac{1}{\delta+1}\right)n+C;
$$
\item there exists a 2-edge-connected graph $F$  with at least $x$ vertices such that 
$$
\mathrm{wsat}(n,F)\leq\frac{\delta}{2}n+C;
$$
\item there exists a connected graph $F$ with odd $\delta$ and $v\geq x$ such that 
$$
\mathrm{wsat}(n,F)\leq\left(\frac{\delta}{2}-\frac{1}{2(\delta+2)}\right)n+C.
$$
\end{enumerate}
\label{th:new_lower_bounds_are_optimal}
\end{theorem}

\proof We will use the following claim (its proof is given after the proof of Theorem~\ref{th:new_lower_bounds_are_optimal}).\\

\begin{claim}
If $\mathrm{wsat}(v,F)=\ell-1$ and $P\subset V(F)$ is such that $|V(F)\setminus P|\geq\delta-1$, then
$$
 \mathrm{wsat}(n,F)\leq\frac{\ell-|E(F|_{V(F)\setminus P})|-1}{|P|}n+O(1).
$$
\label{cl:general_upper_bound_sparse}
\end{claim}

We construct the desired graph sequences for each item of Theorem~\ref{th:new_lower_bounds_are_optimal} separately.

\begin{enumerate}
\item Due to Claim~\ref{cl:general_upper_bound_sparse}, it is sufficient to find a sequence of graphs $F_m$ with $v_m$ vertices, $\ell_m$ edges and minimum degree $\delta\geq 2$ such that $v_m\to\infty$ as $m\to\infty$, each $F_m$ contains $P_m\subset V(F_m)$ such that $|V(F_m)\setminus P_m|\geq\delta-1$ and 
$$
\frac{\ell_m-|E(F|_{V(F_m)\setminus P_m})|-1}{|P_m|}=\frac{\delta}{2}-\frac{1}{\delta+1}.
$$
Consider a disjoint union of $m\geq 3$ cliques $K_{\delta+1}$. $F_m$ is obtained by drawing one edge between one pair of these cliques. Clearly, $F_m$ is the desired sequence with $P_m$ chosen to be one of the disjoint cliques (that does not have an edge joining it with another clique). Indeed, $|V(F_m)\setminus P_m|\geq 2(\delta+1)$, $\ell_m=m\frac{\delta(\delta+1)}{2}+1$, $|E(F|_{V(F_m)\setminus P_m})|=(m-1)\frac{\delta(\delta+1)}{2}+1$.

\item Here we also apply Claim~\ref{cl:general_upper_bound_sparse}: we should find a sequence of 2-edge connected graphs $F_m$ and $P_m\subset V(F_m)$ as above such that
$$
\frac{\ell_m-|E(F|_{V(F_m)\setminus P_m})|-1}{|P_m|}=\frac{\delta}{2}.
$$
Consider a disjoint union of two $(\delta+1)$-cliques $A_1,A_2$ and an $m$-clique $B$, $m\geq\delta+5$. $F_m$ is obtained by 
drawing two disjoint edges $\{u_1,w_1\}$ and $\{u_2,w_2\}$ between $A_1$ and $B$ ($u_1,u_2$ are vertices of $A_1$), two disjoint edges between $A_2$ and $B$ that also do not meet vertices $w_1,w_2$, and deleting the edge $\{u_1,u_2\}$. It is easy to verify that $\mathrm{wsat}(v_m=m+2(\delta+1),F_m)=\ell_m-1$, where $\ell_m=\frac{m(m-1)}{2}+\delta(\delta+1)+3$. The desired set $P_m$ is the set of vertices of $A_1$. Indeed, $|V(F_m)\setminus P_m|=m+\delta+1$, $|E(F|_{V(F_m)\setminus P_m})|=\frac{m(m-1)}{2}+\frac{\delta(\delta+1)}{2}+2$.

\item In the same way as above, we should find a sequence of connected graphs $F_m$ with odd $\delta\geq 3$ and $P_m\subset V(F_m)$ such that
$$
\frac{\ell_m-|E(F|_{V(F_m)\setminus P_m})|-1}{|P_m|}=\frac{\delta}{2}-\frac{1}{2(\delta+2)}.
$$
Consider a disjoint union of two $(\delta+1)$-cliques $A_1,A_2$ and an $m$-clique $B$, $m\geq\delta+3$, with distinguished vertices $w_1\neq w_2$. $F_m$ is obtained by 
\begin{itemize}
\item for every $j\in\{1,2\}$, choosing a perfect matching in $A_j$ arbitrarily and deleting it from $A_j$,
\item for every $j\in\{1,2\}$, selecting a vertex $u_j\in V(A_j)$ and drawing an edge between $u_j$ and $w_j$,
\item adding two vertices $x_1,x_2$ and drawing edges between $x_2$ and all vertices of $V(A_2)$, between $x_1$ and all vertices from $V(A_1)\setminus\{u_1\}$. 
\end{itemize}
Note that $\mathrm{wsat}(v_m=m+2(\delta+1)+2,F_m)=\ell_m-1$, where $\ell_m=\frac{m(m-1)}{2}+(\delta+1)^2+1$. Indeed, the missing edges can be added one by one, say, in the following order: start with missing $\{x_1,u_1\}$, then draw all edges between $V(A_1)\cup V(A_2)$ and $V(B)\setminus\{w_1,w_2\}$, proceed with the missing matchings in $A_1,A_2$, then draw edges between $\{x_1,x_2\}$ and $V(B)$, draw all the remaining edges between $V(A_1)\cup V(A_2)$ and $V(B)$, and, finally, restore the missing edges between $V(A_1)\cup\{x_1\}$ and $V(A_2)\cup\{x_2\}$. It remains to set $P_m=V(A_1)\cup\{x_1\}$. Indeed, $|V(F_m)\setminus P_m|=m+\delta+2$, $|E(F|_{V(F_m)\setminus P_m})|=\frac{m(m-1)}{2}+\frac{(\delta+1)^2}{2}+1$. $\Box$\\

\end{enumerate}

{\it Proof of Claim~\ref{cl:general_upper_bound_sparse}}. Denote $k=|P|$. We are going to show by induction that, for every $m\in\mathbb{Z}_{\geq 0}$, 
\begin{equation}
\mathrm{wsat}(v+km,F)\leq (\ell-|E(F|_{V(F)\setminus P})|-1)m+\ell-1.
\label{eq:general_upper_bound_sparse_proof}
\end{equation}
Note that (\ref{eq:general_upper_bound_sparse_proof}) immediately implies the statement of Claim~\ref{cl:general_upper_bound_sparse} since, if $r$ is a reminder of the division of $n-v$ by $k$, then $(\delta-1)r=O(1)$ edges is sufficient to restore all edges of $K_n$ from $K_{v+km}$.

The base of induction is straightforward. Let $m$ be a positive integer and assume that (\ref{eq:general_upper_bound_sparse_proof}) is proven for $m-1$. Let $H\in\underline{\mathrm{wSAT}}(v+k(m-1),F)$. Let $G$ be obtained from $H$ by adding a copy of $F|_P$ to $H$ (on a set of $k$ vertices $\tilde P$ disjoint with $V(H)$), distinguishing a subset $K\subset V(H)$ of $v-k$ vertices, drawing edges between $K$ and $\tilde P$ in the same way as they appear between $V(F)\setminus P$ and $P$, and deleting one of these edges. All missing edges between $K$ and $\tilde P$ in $G$ can be restored since $\mathrm{wsat}(v,F)=\ell-1$. After that, every vertex from $\tilde P$ has at least $|K|\geq\delta-1$ neighbors in $V(H)$. Therefore, all edges between $\tilde P$ and $V(H)\setminus K$ can be restored as well. This finishes the proof.  $\Box$

\section{Non-concentration}
\label{sc:non_concentration}

In this section, we prove Theorem~\ref{th:non_concentration}.

Let $\delta\geq 2$ and $k\leq\frac{(\delta-2)(\delta+1)}{2}$ be a non-negative integer. The proof strategy is similar to the proof of Theorem~\ref{th:new_lower_bounds_are_optimal}. We are going to construct a sequence of graphs $F$ such that $\gamma$ defined in Section~\ref{sc_sub:gamma} coincides with the minimum value of $\frac{\ell-|E(F|_{V(F)\setminus P})|-1}{|P|}$ over all $P$ satisfying the requirements in Claim~\ref{cl:general_upper_bound_sparse}, and this minimum value equals to $\frac{\delta}{2}+\frac{k}{\delta+1}=:\rho_k$. We start with constructing the desired graph sequence.

Fix the sequence $s_1\leq s_2\leq\ldots\leq s_{\delta+1}$ (it clearly exists and it is unique) such that $s_1=s_2=0$, for every $i\in[\delta]$, $s_{i+1}-s_i\leq 1$, $\sum_{i=1}^{\delta+1}s_i=k+1$. Let $m\geq\max\{2(k+1),2\delta+1\}$	 be an integer. Consider disjoint cliques $A\cong K_{\delta+1}$ and $B\cong K_m$. Let $[\delta+1]$ be the vertex set of $A$. Let $M_1,\ldots,M_{\delta+1}$ be disjoint subsets of $V(B)$ satisfying $|M_i|=s_i$, $i\in[\delta+1]$. The graph $F_m$ is obtained from $A$ and $B$ by drawing edges between the vertex $i$ and every vertex from $M_i$ for all $i\in[\delta+1]$. Let $P_m:=V(A)$. The graph $F_m$ has $v_m=m+\delta+1$ vertices and $\ell_m={m\choose 2}+{\delta+1\choose 2}+k+1$ edges. Note that $|V(F_m)\setminus P_m|=m\geq 2\delta+1$ and 
$$
\frac{\ell_m-|E(F_m|_{V(F_m)\setminus P_m})|-1}{|P_m|}=\frac{{\delta+1\choose 2}+k}{\delta+1}=\frac{\delta}{2}+\frac{k}{\delta+1}=\rho_k.
$$
Due to the definition of $\gamma$ and its properties described in Section~\ref{sc_sub:gamma}, by Theorem~\ref{th:main_lower_subadditive} and Claim~\ref{cl:general_upper_bound_sparse}, it remains to prove the following:
\begin{enumerate}
\item $\mathrm{wsat}(v_m,F_m)=\ell_m-1$;
\item for every $Q\subset V(F_m)$, we have 
\begin{equation}
\frac{\ell_m-|E(F_m|_{V(F_m)\setminus P_m})|-1}{|P_m|}\leq\frac{\ell_m-|E(F_m|_{V(F_m)\setminus Q})|-1}{|Q|}.
\label{eq:P_min}
\end{equation}
\end{enumerate}

The first condition holds since the missing edges of $K_{v_m}$ can be added one by one, say, in the following order: first join vertices $1,2$ (note that this vertices do not have neighbors in $B$ initially) with all the vertices of $V(B)\setminus(\cup_i M_i)$, then join all the other vertices of $A$ with all the vertices of $V(B)\setminus(\cup_i M_i)$, and, finally, restore all the rest.

Now, fix $Q\subset V(F_m)$, $|Q|=x$. Let us prove the inequality (\ref{eq:P_min}). 

Assume that $Q\subset V(A)$. We have
\begin{equation}
\frac{\ell_m-|E(F_m|_{V(F_m)\setminus Q})|-1}{|Q|}=
\frac{x(\delta-\frac{x-1}{2})+\sum_{i\in Q}s_i-1}{x}\geq
\frac{x(\delta-\frac{x-1}{2})+\sum_{i=1}^x s_i-1}{x}.
\label{eq:non_concentration_p_m_optimal}
\end{equation}
while
$$
\frac{\ell_m-|E(F_m|_{V(F_m)\setminus P_m})|-1}{|P_m|}=
\frac{(\delta+1)(\delta-\frac{\delta}{2})+\sum_{i=1}^{\delta+1} s_i-1}{\delta+1}.
$$
It remains to notice that the right hand part of (\ref{eq:non_concentration_p_m_optimal}) (we denote it by $\xi(x)$) decreases in $x$. Indeed,
$$
 [\xi(x+1)-\xi(x)]x(x+1)=-\frac{x(x+1)}{2}+1-\sum_{i=1}^x s_i+xs_{x+1}=
 \sum_{i=1}^x(s_{x+1}-s_i)+1-\frac{x(x+1)}{2}\leq 0.
$$
If $Q\subset V(B)$, then 
$$
\frac{\ell_m-|E(F_m|_{V(F_m)\setminus Q})|-1}{|Q|}\geq\frac{x\frac{m-1}{2}-1}{x}\geq\frac{m-1}{2}-1\geq\delta-1\geq\rho_k.
$$
Finally, if $Q$ has non-empty intersections both with $A$ and $B$, then, letting $W=Q\cap V(A)$, we get
\begin{align*}
\frac{\ell_m-|E(F_m|_{V(F_m)\setminus Q})|-1}{|Q|} & \geq
\frac{\ell_m-|E(F_m|_{V(F_m)\setminus W})|-1+(|Q|-|W|)\frac{m-1}{2}}{|Q|}\\
& \geq
\frac{|W|\rho_k+(|Q|-|W|)\frac{m-1}{2}}{|Q|}>\frac{|W|\rho_k+(|Q|-|W|)\rho_k}{|Q|}=\rho_k.
\end{align*}

\section{Tight results for other graph families}
\label{sc:other}

\subsection{The conjecture of Faudree, Gould and Jacobson}
\label{sc:faudree_conjecture}

Here we prove the conjecture of Faudree, Gould and Jacobson~\cite{Faudree}. Let $F_{v,\delta}$ be obtained from $K_v$ by removing $v-1-\delta$ edges adjacent to the same vertex. 
\begin{theorem}
For all integer $v\geq 3$ and $\delta\in[v-2]$, 
$$
\mathrm{wsat}(n,F_{v,\delta})={v-1\choose 2}+(n-v+1)(\delta-1).
$$
\label{th:F_pq}
\end{theorem}

\proof Recall that, for every graph $F$, $\mathrm{wsat}(n,F)\leq(\delta-1)(n-v)+\mathrm{wsat}(v,F)$ (see~(\ref{wsat:upper_bound_general}) in Introduction). Since $\mathrm{wsat}(v,F_{v,\delta})=\ell-1={v\choose 2}-(v-\delta)$, it immediately gives the upper bound.

Now we prove the lower bound. Let $H\in\underline{\mathrm{wSAT}}(n,F_{v,\delta})$. Consider an $F_{v,\delta}$-bootstrap percolation process that starts on $H$ and stops on $K_n$. Note that every new edge in this process creates a copy of $F_{v,\delta}$ and thus creates a copy of $K_{\delta+1}$. Let $F_0$ be a copy of $F_{v,\delta}$ that is created together with the first edge $e$ added to $H$. Let $x$ be the vertex of $F_0$ of degree $\delta$, and $X$ be a $(\delta-1)$-subset in its neighborhood in $F_0\setminus\{e\}$. Delete from $H$ all the edges of $F_0\setminus\{e\}$ that have both vertices outside $X$. 
Clearly, the obtained graph is weakly $K_{\delta+1}$-saturated in $K_n$. Therefore,
$$
 \mathrm{wsat}(n,K_{\delta+1})\leq\mathrm{wsat}(n,F_{v,\delta})-{v-\delta\choose 2}.
$$
Then the equality $\mathrm{wsat}(n,K_{\delta+1})=n(\delta-1)-{\delta\choose 2}$ (see~\cite{L77}) immediately gives the desired lower bound. $\Box$\\

\begin{remark}
If $F$ is an arbitrary (not necessarily disjoint) union of any number of cliques of size at least $\delta+1$, then the same argument as in the proof of Theorem~\ref{th:F_pq} implies that $\mathrm{wsat}(n,F)=n(\delta-1)+O(1)$ certifying the positive answer to the question about possible values of $c_F$ asked in Introduction. We shall also note that Theorem~\ref{th:F_pq} follows immediately from (\ref{wsat:upper_bound_general}), Claim~\ref{cl:wsat_lower_bound_additive_constant} and the inequality $c_F\geq\delta-1$ that holds true since every edge appearing in an $F$-bootstrap percolation process creates a copy of $K_{\delta+1}$. However, we decided to present the above proof since it shows another neat and quite general approach (a similar idea was used in~\cite[Theorem 2]{FG}). Moreover, Claim~\ref{cl:wsat_lower_bound_additive_constant} also implies that the upper bound (\ref{wsat:upper_bound_general}) is tight for all  $F$ being unions of cliques of size at least $\delta+1$ with $\mathrm{wsat}(v,F)=\ell-1$.
\label{rk:union_of_cliques}
\end{remark}

\subsection{General tight bounds}
\label{sc:general_tight}

In this section, we present an improvement of Theorem~\ref{th:main_lower_subadditive} that gives the exact value of the weak saturation number for a certain family of graphs $F$.\\

For $r\in\mathbb{Z}_{\geq 0}$, define
$$
g^*_r(0)=0,\quad g^*_r(i)=\min_{s\in[i],\,1\leq i_1,\ldots,i_s\leq v-r:\,i_1+\ldots+i_s=i}(e_{i_1}+\ldots+e_{i_s}).
$$
Thus, $g^*=g^*_1$. Note that, for every fixed $i$, $g^*_r(i)$ is a nondecreasing function of $r$. The following improvement of Theorem~\ref{th:main_lower_subadditive} is valid when there is a small enough weakly $F$-saturated graph.

\begin{theorem}
If, for all $n\geq v$, $\mathrm{wsat}(n,F)\leq g^*_2(n-v)+\ell-1$, then, for all $n\geq v$, $\mathrm{wsat}(n,F)=g^*_2(n-v)+\ell-1$.
\label{th:g_2}
\end{theorem}

We prove Theorem~\ref{th:g_2} in Section~\ref{sc:proof-th:g_2}. Note that it immediately implies that $\mathrm{wsat}(n,K_4)=2n-3$ (though a combinatorial proof of this fact was known~\cite{bol}). A family of graphs $F$ such that Theorem~\ref{th:g_2} gives an exact value of the weak saturation number can be distilled using Theorem~\ref{th:bridges} which we state below and prove in Section~\ref{sc:proof-th:bridges}. 

Consider the set $\mathcal{K}$ of all $i\in[v]$ such that there are no positive integers $i_1,\ldots,i_s$, $s\geq 2$, satisfying $i_1+\ldots+i_s=i$ and $e_i\geq e_{i_1}+\ldots+e_{i_s}$. Let $\mathcal{K}_r=\mathcal{K}\cap[v-r]$ for $r\in\{0,1,\ldots,v-1\}$. Let $\beta=\beta(F)$ be the minimum number of vertices that can be deleted from $F$ in such a way that the remaining graph contains a cut-edge (i.e. its deletion increases the number of connected components). 

\begin{theorem}
Let $r\in\{0,1,\ldots,v-1\}$. Then the following three properties are equivalent:
\begin{enumerate}
\item For every $n\geq v$, $\mathrm{wsat}(n,F)\leq g^*_r(n-v)+\ell-1$.
\item $\mathrm{wsat}(v,F)=\ell-1$, and the maximum element of $\mathcal{K}_r$ is at most $v-\beta$.
\item $\mathrm{wsat}(v,F)=\ell-1$, and, for every $i\geq 0$, $g^*_{\beta}(i)\leq g^*_r(i)$.
\end{enumerate}
\label{th:bridges}
\end{theorem}

Note that, in particular, if we let $r=\beta$, then we immediately get the following corollary.
\begin{corollary}
\label{crl:beta}
If $\mathrm{wsat}(v,F)=\ell-1$, then, for every $n\geq v$, $\mathrm{wsat}(n,F)\leq g^*_{\beta}(n-v)+\ell-1$.
\end{corollary}
Let us stress that any other possible upper bound of the form $g^*_r(n-v)+\ell-1$ could not be better than the bound provided in Corollary~\ref{crl:beta}. Moreover,
\begin{corollary}
If $\mathrm{wsat}(v,F)=\ell-1$ and the maximum element of $\mathcal{K}_2$ is at most $v-\beta$, then, for every $n\geq v$, $\mathrm{wsat}(n,F)=g^*_2(n-v)+\ell-1$.
\label{cry:small_K_2}
\end{corollary}

We immediately get that $\mathrm{wsat}(n,F)=g^*_2(n-v)+\ell-1$ for {\it all connected graphs $F$ that are not 4-edge-connected and satisfy} $\mathrm{wsat}(v,F)=\ell-1$, since for such graphs $\beta\leq 2$ due to the following claim. 

\begin{claim}
Let $k$ be the edge-connectivity of $F$. Then $\beta\leq k-1$.
\label{cl:edge_connectivity_to_beta}
\end{claim}

\begin{proof}
Let $\mathcal{E}$ be a set of $k$ edges such that their deletion makes $F$ disconnected. Let $\{a,b\}\in\mathcal{E}$. Let us follow the edges of $\mathcal{E}\setminus\{a,b\}$ one by one and delete from $F$, at each step, a single vertex of the considered edge other than both $a$ and $b$ (if such a vertex is already deleted, then we just move to the next edge without any deletion). Eventually we get a graph with the cut-edge $\{a,b\}$. The number of vertices that were deleted is at most $k-1$.
\end{proof}

We are also able to get the sharp value for a family of graphs $F$ that are obtained by drawing several {\it disjoint} edges between two cliques. For positive integers $1\leq c\leq a\leq b$, let $F_{a,b,c}$ be obtained by drawing $c$ disjoint edges between disjoint $K_a$ and $K_b$. It is clear that $\mathrm{wsat}(v,F)=\ell-1$. Therefore, the exact value for $c\leq 3$ follows from Corollary~\ref{cry:small_K_2} and Claim~\ref{cl:edge_connectivity_to_beta}: $\mathrm{wsat}(n,F_{a,b,c})=g^*_2(n-v)+\ell-1$, where $g^*_2$ can be easily computed directly. In particular, for $i$ divisible by $a$, we get $g^*_2(i)=\frac{i}{a}\left({a\choose 2}+c-1\right)$. It could be generalised to larger values of $c$.

\begin{theorem}
If $a<b$, then, for every $n\geq v$, $\mathrm{wsat}(n,F_{a,b,c})=g^*(n-v)+\ell-1$. If $a=b=5$, $c=4$, then  $\mathrm{wsat}(n,F_{a,b,c})=g^*_2(n-v)+\ell-1$.
\label{th:F_a_b_c}
\end{theorem} 

\begin{proof}
First, let $a<b$. Due to Theorem~\ref{th:bridges}, it is sufficient to prove that the maximum element of $\mathcal{K}_1$ is at most $a$, since $a\leq v-\beta$.  Assume the contrary: there exists $i\in[a+1,a+b-1]$ such that $i\in\mathcal{K}_1$. Let $\tilde a\leq a$ and $\tilde b\leq b$ be such that the union of some $\tilde a$ vertices from the $K_a$-part of $F_{a,b,c}$ with some $\tilde b$ vertices from the $K_b$-part has exactly $i$ vertices and $e_i+1$ edges with endpoints in this union. Since $\tilde a+\tilde b=i$ and $\tilde a\leq a$, we get $\tilde b>0$.

If $\tilde b<a$, then 
$$
e_i\geq \tilde b(b-1)-{\tilde b\choose 2}+e_{\tilde a}=
\tilde b\left(b-1-\frac{\tilde b-1}{2}\right)+e_{\tilde a}\geq
\tilde b\left(a-\frac{\tilde b-1}{2}\right)+e_{\tilde a}>e_{\tilde b}+e_{\tilde a}
$$
--- a contradiction.

If $a\leq \tilde b<b$, then consider integers $s\geq 1$ and $0\leq r<a$ such that $\tilde b=sa+r$. We shall prove that $e_i\geq e_{\tilde a}+se_a+e_r$. In the same way as above, it is sufficient to show that $\tilde b(b-1)-{\tilde b\choose 2}\geq se_a+e_r$. Note that $e_a\leq {a\choose 2}+a-1$, $e_r\leq r a-{r\choose 2}$. We get
\begin{align*}
 \tilde b(b-1)-{\tilde b\choose 2} & =(sa+r) \left(b-\frac{\tilde b+1}{2}\right)\geq
 (sa+r)\frac{\tilde b+1}{2}\\
 &\geq(sa+r)\frac{a+r+1}{2}\geq
 se_a+s+\frac{sar}{2}+r\frac{a+r+1}{2}>se_a+ar\geq s e_a+e_r
\end{align*}
as needed.

Finally, let $\tilde b=b$. Then $\tilde a<a$. We let $\tilde b+\tilde a=sa+r$, where $s\geq 1$ and $0\leq r<a$. Note that $e_a\leq{a\choose 2}+c-1$ and $e_r\leq r(a-1)-{r\choose 2}+\min\{r,c\}$. We get
\begin{align*}
 e_i &={b\choose 2}+c+\tilde a(a-1)-{\tilde a\choose 2}-1
 =\frac{(sa+r-\tilde a)(sa+r-\tilde a-1)}{2}+c+\tilde a(a-1)-{\tilde a\choose 2}-1\\
 &\geq se_a+e_r+\frac{s^2-s}{2}a^2+s(ar-a\tilde a-c+1)+r^2-r(a+\tilde a)+a\tilde a+c-\min\{r,c\}-1.
\end{align*}
If $s=1$, then $r=\tilde b+\tilde a-a\geq \tilde a+1$. In this case,
$$
 e_i\geq se_a+e_r+r^2-\tilde ar-\min\{r,c\}\geq
 se_a+e_r+r-\min\{r,c\}\geq se_a+e_r.
$$
If $s\geq 2$, then
\begin{align*}
 e_i & \geq se_a+e_r+
(s-1)a^2+r^2+sar-(s-1)a\tilde a-\tilde ar-sc+s-ra+c-r-1\\
&\geq se_a+e_r+(s-1)a+r^2-sc+s+c-r-1\\
&\geq se_a+e_r+r^2+s-r-1\geq se_a+e_r.
\end{align*}
The first part of Theorem~\ref{th:F_a_b_c} follows.\\

Let $a=b=5$, $c=4$. Computing directly all $e_i$, $i\leq 7$, we get that the maximum element in $\mathcal{K}_2$ equals $5$. Since $\beta=3$, the second part of Theorem~\ref{th:F_a_b_c} follows from Theorem~\ref{th:bridges}.
\end{proof}

Unfortunately, we can not generalise Theorem~\ref{th:F_a_b_c} to the case $a=b>5$, $c\geq 4$, or $a=b=c=5$. Also, note that, when $a<b$, we get the upper bound $g^*_1$, and not $g^*_2$, as everywhere before in this section. Actually, $g^*_1=g^*_2$ in this case since $g^*_r$ is non-decreasing. Indeed, if $g^*_1(i)<g^*_2(i)$ for some $i$, then $\mathrm{wsat}(v+i,F_{a,b,c})<g^*_2(i)+\ell-1$ whereas $\mathrm{wsat}(n,F_{a,b,c})\leq g^*_2(n-v)+\ell-1$ for all $n$, that contradicts Theorem~\ref{th:g_2}.\\

We shall conclude this section by noting that it is not always true that $\mathrm{wsat}(n,F)= g^*_{\beta}(n-v)+\ell-1$ even when $\mathrm{wsat}(v,F)=\ell-1$. To see this, consider $F$ obtained by the deletion of a perfect matching (consisting of 4 edges) from $K_9$. It is obvious that $\mathrm{wsat}(v=9,F)=\ell-1=31$, that $\beta=6$ and that $g^*_3(i)=\frac{17}{3}i+O(1)$. Let us now show that actually $\mathrm{wsat}(n,F)\leq\frac{11}{2}n+O(1)$ implying $\mathrm{wsat}(n,F)<g^*_{\beta}(n-v)+\ell-1$ for $n$ large enough.

Consider a 4-set $P\subset V(F)$ of single ends of all 4 edges of the missing matching in $F$. Clearly $P$ induces a clique, and there are exactly 22 edges in $F$ adjacent to $P$. Let us show that from any weakly saturated $H$ we may get another weakly saturated graph $\tilde H$ by adding 4 vertices and 22 edges, that clearly implies the desired claim. Let $\tilde H$ be obtained from $H$ by adding the 4-clique $P$ together with the 16 edges going from $P$ to some 5 vertices $v_1,\ldots,v_5$ in $H$ exactly as in $F$. Delete one of the edges with both ends in $P$, and add an edge from $P$ to some $v_6\in V(H)\setminus\{v_1,\ldots,v_5\}$ instead. It is easy to see that the final graph $\tilde H$ is indeed weakly saturated. First of all, the deleted edge in $P$ and all missing edges between $P$ and $v_1,\ldots,v_5$ can be added since $\mathrm{wsat}(v,F)=\ell-1$. Secondly, we may add all edges from $v_6$ to $P$, and then add all the other edges.

\subsection{Proof of Theorem~\ref{th:g_2}}
\label{sc:proof-th:g_2}

The proof strategy is actually similar to those in the proof of Theorem~\ref{th:main_lower_subadditive}, the main difference is that we will not force sets of vertices $B_{\kappa}$ to be disjoint. However, instead we will require sets of edges induced by $B_{\kappa}$ to be disjoint. That would actually imply that each pair of vertex sets has at most 1 vertex in common. 

So, let $n\geq v$. For $i\geq v$, set $f_2(i)=g_2^*(i-v)+\ell-1$. Let $H\in\underline{\mathrm{wSAT}}(n,F)$. Let $\mathcal{O}_H$ be the set of all vectors $(B_1,\ldots,B_k)$ such that
\begin{itemize}
\item for every $\kappa\in[k]$, $B_{\kappa}\subset[n]$, $|B_{\kappa}|\geq v$,
\item for $\kappa_1\neq\kappa_2$, $E(H|_{B_{\kappa_1}})\cap E(H|_{B_{\kappa_2}})=\varnothing$,
\item for every $\kappa\in[k]$, $|E(H|_{B_{\kappa}})|\geq f_2(|B_{\kappa}|)$,
\item $|B_1|\geq|B_2|\geq\ldots\geq|B_k|$.
\end{itemize}

Note that the inequality $|E(H|_{B_{\kappa}})|\geq f_2(|B_{\kappa}|)$ immediately implies that $|E(H|_{B_{\kappa}})|=f_2(|B_{\kappa}|)$ since, by the assumption of Theorem~\ref{th:g_2}, $f_2(|B_{\kappa}|)$ edges is enough to saturate a clique on $B_{\kappa}$. Moreover, for $\kappa_1\neq\kappa_2$, $|B_{\kappa_1}\cap B_{\kappa_2}|\leq 1$, since otherwise we may find a weakly $F$-saturated $\tilde H$ in $K_n$ with the number of edges less than in $H$. Indeed, since there are no edges in $B:=B_{\kappa_1}\cap B_{\kappa_2}$, we may renew the $f_2(|B_{\kappa_1}|)$ edges induced by $B_{\kappa_1}$ in such a way that at least 1 edge is entirely inside $B$, and the new set of edges on $B_{\kappa_1}$ is weakly $F$-saturated in the clique $K_1$ on $B_{\kappa_1}$. Next, in the same way, it is also possible to renew the set of $f_2(|B_{\kappa_2}|)$ edges induced by $B_{\kappa_2}$ in such a way that at least 1 edge is entirely inside $B$, and this set of edges $E_2$ is weakly $F$-saturated in the clique $K_2$ on $B_{\kappa_2}$. Add to the constructed at the previous step graph all the edges from $E_2$ that are not entirely in $B$. We get the graph with less than $f_2(|B_{\kappa_1}|)+f_2(|B_{\kappa_2}|)$ edges which is weakly $F$-saturated in $K_1\cup K_2$ --- a contradiction with the minimality of $H$.

Let us also note that $\mathcal{O}_H$ is obviously non-empty since we may take a single $B$ which is the vertex set of the first copy of $F$ that appears in the bootstrap percolation process that starts at $H$. We order $\mathcal{O}_H$ lexicographically in the same way as in the proof of Theorem~\ref{th:main_lower_subadditive}, and define $\mathcal{B}^*_H$ as a maximal element of $\mathcal{O}_H$. Thus, it remains to prove the following lemma.

\begin{lemma}
$\mathcal{B}^*:=\max_{H\in\underline{\mathrm{wSAT}}(n,F)}\mathcal{B}^*_H=([n])$.
\label{lm:max_decomposition_2}
\end{lemma}

\begin{proof}
Assume the contrary. Let $H$ be a weakly $F$-saturated subgraph of $K_n$ such that $\mathcal{B}^*_H=\mathcal{B}^*$. Let us first prove that there exists a pair of non-adjacent in $H$ vertices $u$ and $v$ that do not belong to any common $B_{\kappa}$ from $\mathcal{B}^*$. Let $B:=B_1$ be the first set in $\mathcal{B}^*$. Due to the assumption, there exists a vertex $v$ outside $B$. Assume that every $u\in B$ is either adjacent to $v$ in $H$, or belongs together with $v$ to the same $B_u$ from $\mathcal{B}^*$. Take any $u\in B$ non-adjacent to $v$ and satisfying the second condition. As above, we may renew edges inside $B_u$ in order to make $u$ and $v$ adjacent while keeping the same amount of edges inside $B_u$. Note that we have not changed adjacencies in all the other $B_{\kappa}$ from $\mathcal{B}^*$ since any two sets share at most 1 vertex. Thus, proceeding in this way, we may actually make $v$ adjacent to all vertices of $B$. Denote the renewed graph by $\tilde H$. It is clear from the definition of $g_2^*$ that $g_2^*(|B|+1-v)\leq g_2^*(|B|-v)+|B|$ implying that $f_2(|B|+1)\leq f_2(|B|)+|B|$. Then, $(B\cup\{v\})\in\mathcal{O}_{\tilde H}$ --- contradiction with the maximality of $\mathcal{B}^*$.

We then take a pair of non-adjacent vertices $u$ and $v$ that do not belong to any common $B_{\kappa}$ from $\mathcal{B}^*$ and that is activated first in a bootstrap percolation process initiated at $H$. Let $\tilde F$ be a copy of $F$ that appears together with $\{u,v\}$ in this process. Consider the set $S$ of all $\kappa$ such that $|B_{\kappa}\cap V(\tilde F)|\geq 2$. Note that $S$ is non-empty since otherwise we may add $V(\tilde F)$ to $\mathcal{B}^*$ --- contradiction with maximality. Let $q\in S$ be such that $|B_q|$ is maximal. In the usual way, we may renew edges in every $B_{\kappa}$, $\kappa\in S\setminus\{q\}$, so that $B_{\kappa}\cap V(\tilde F)$ induces in $H$ exactly the same graph as in $\tilde F$. Thus, we may assume that all edges of $\tilde F$ other than $\{u,v\}$ and those that belong to $B_q$ initially belong to $H$. Note that the set of edges induced by $B:=B_q\cup V(\tilde F)$ equals the union of the set of edges induced by $B_q$ and $E(\tilde F)\setminus\{u,v\}$ since otherwise there exists a graph on $B$ which is weakly $F$-saturated and has less number of edges --- contradiction with the minimality of $H$.

Let us show that the tuple $\mathcal{B}$, composed of sets $B_{\kappa}$, $\kappa\notin S$, and $B$ placed in the right order, belongs to $\mathcal{O}_H$, contradicting the maximality of $\mathcal{B}^*$. The only not so trivial condition from the definition of $\mathcal{O}_H$ is the third one (in particular, note that the second one holds since any of $B_{\kappa}$, $\kappa\notin S$, does not contain edges from $E(\tilde F)\setminus\{u,v\}$). 
 It remains thus to verify this third condition. Denote $i:=v-|B_q\cap V(\tilde F)|$. Then
\begin{align*}
 |E(H|_B)| & \geq |E(H|_{B_q})|+e_i\\
 &\geq f_2(|B_q|)+e_i=
 g_2^*(|B_q|-v)+e_i+\ell-1\geq g_2^*(|B_q|+i-v)+\ell-1= f_2(|B|)
\end{align*}
completing the proof.
\end{proof}

\subsection{Proof of Theorem~\ref{th:bridges}}
\label{sc:proof-th:bridges}

We need several claims. First of all, let us prove that $\mathrm{wsat}(i+v,F)-(\ell-1)$ is subadditive.

\begin{claim}
For every graph $F$ and every $i,j\in\mathbb{Z}_{\geq 0}$, we have that 
$$
 \mathrm{wsat}(i+j+v,F)-(\ell-1)\leq [\mathrm{wsat}(i+v,F)-(\ell-1)]+[\mathrm{wsat}(j+v,F)-(\ell-1)].
$$
\label{cl:wsat-subadditive}
\end{claim}

\begin{proof}
Let $A\in\underline{\mathrm{wSAT}}(v+i,F)$, $B\in\underline{\mathrm{wSAT}}(v+j,F)$. Let $U$ be a set of vertices in $A$ of size $v$, and let $\tilde F$ be the first copy of $F$ that appears in an $F$-bootstrap percolation process initiated at $B$. Note that the number of edges in $B|_{V(\tilde F)}$ is at least $\ell-1$.

Consider the graph $H$ on $V(A)\cup[V(B)\setminus V(\tilde F)]$ constructed as follows. Add to $A$ disjointly the set of vertices $W:=V(B)\setminus V(\tilde F)$ and preserve those edges of $B$ that have at least one end in $W$ in the following way. The edges that are entirely inside $W$ just remain the same, and the edges between $W$ and $V(\tilde F)$ are moved to form a bipartite graph between $W$ and $U$: consider any bijection $\varphi:V(\tilde F)\to U$, and draw an edge $\{x\in W,\varphi(y)\}$ if and only if $\{x,y\}\in E(B)$. It is obvious that $H$ is a weakly $F$-saturated graph with at most $\mathrm{wsat}(i+v,F)+\mathrm{wsat}(j+v,F)-(\ell-1)$ edges, completing the proof.
\end{proof}

Next we claim that, for a subadditive function $g(i)$ not to exceed $g^*_r(i)$ for all $i$, it is sufficient to satisfy $g(i)\leq g^*_r(i)$ for $i\in\mathcal{K}_r$. 

\begin{claim}
Let $g:\mathbb{Z}_{\geq 0}\to\mathbb{R}$ be a subadditive function such that $g(0)=0$. If $g(i)\leq g^*_r(i)$ for all $i\in\mathcal{K}_r$, then the same inequality holds true for all $i\in\mathbb{Z}_{\geq 0}$.
\label{cl:from_K_r_to_all_Z}
\end{claim}

\begin{proof}
Note that $g(0)=g^*_r(0)$ be the definition of $g^*_r$. Let us then take any $i\in\mathbb{N}$, $i\notin\mathcal{K}_r$. First of all, by the definition of $g^*_r$, there exist $i_1,\ldots,i_h\in[v-r]$ such that $i_1+\ldots+i_h=i$ and $g^*_r(i)=e_{i_1}+\ldots+e_{i_h}$. Secondly, by the definition of $\mathcal{K}_r$, for every $j\in[h]$ such that $i_j\notin\mathcal{K}_r$, there exist $i^j_1,\ldots,i^j_{s_j}\in\mathcal{K}_r$ such that $i^j_1+\ldots+i^j_s=i_j$ and $e_{i_j}\geq e_{i^j_1}+\ldots+e_{i^j_s}$. We conclude that there exist $i^1,\ldots,i^s\in\mathcal{K}_r$ such that $i=i^1+\ldots+i^s$ and 
$$
g^*_r(i)\geq e_{i^1}+\ldots+e_{i^s}\geq g^*_r(i^1)+\ldots+g^*_r(i^s)\geq g(i^1)+\ldots+g(i^s)\geq g(i),
$$
as needed.
\end{proof}

Let $r_1,r_2\in[v-1]$. Assume that the maximum element of $\mathcal{K}_{r_1}$ is at most $v-r_2$. It immediately implies that $\mathcal{K}_{r_1}\subseteq\mathcal{K}_{r_2}$. Then, for all $i\in\mathcal{K}_{r_1}$, $g^*_{r_2}(i)=e_i$. Therefore, for all $i\in\mathbb{Z}_{\geq 0}$, $g^*_{r_2}(i)\leq g^*_{r_1}(i)$ due to Claim~\ref{cl:from_K_r_to_all_Z}. On the other hand, assuming that $g^*_{r_2}(i)\leq g^*_{r_1}(i)$ for all $i\in\mathbb{Z}_{\geq 0}$ and that $i>v-r_2$ for some $i\in\mathcal{K}_{r_1}$, we get that there exist $i_1,\ldots,i_s\in[v-r_2]$ satisfying $i_1+\ldots+i_s=i$ and $g^*_{r_2}(i)=e_{i_1}+\ldots+e_{i_s}$. But then $e_i\geq g^*_{r_1}(i)\geq e_{i_1}+\ldots+e_{i_s}$ contradicting the fact that $i\in\mathcal{K}_{r_1}$. Setting $r_1=r$ and $r_2=\beta$, we immediately get the equivalence of the second and the third properties in Theorem~\ref{th:bridges}.

To prove the equivalence of the first and the second property, let us fix a graph $F$ with $\mathrm{wsat}(v,F)=\ell-1$,  and note that Claim~\ref{cl:wsat-subadditive} and Claim~\ref{cl:from_K_r_to_all_Z} imply that the inequality $\mathrm{wsat}(n,F)\leq g^*_r(n-v)+\ell-1$ is true for every $n$ if and only if the inequality $\mathrm{wsat}(i+v,F)\leq e_i+(\ell-1)$ is true for ever $i\in\mathcal{K}_r$. Then, the following lemma finishes the proof of Theorem~\ref{th:bridges}.

\begin{lemma}
Let $F$ be a graph such that $\mathrm{wsat}(v,F)=\ell-1$. Then $\mathcal{K}_{\beta}=\{i\in\mathcal{K}:\,\mathrm{wsat}(i+v,F)\leq e_i+(\ell-1)\}$.
\end{lemma}

\begin{proof}
Let us first proof Corollary~\ref{crl:beta} without relying to the statement of Theorem~\ref{th:bridges}. It would imply $\mathcal{K}_{\beta}\subseteq\{i\in\mathcal{K}:\,\mathrm{wsat}(i+v,F)\leq e_i+(\ell-1)\}$.

So, we fix $n\geq v$ and construct a weakly $F$-saturated graph $H$ with $g^*_{\beta}(n-v)+\ell-1$ edges. First of all, let us find $i_1,\ldots,i_s\in[v-\beta]$ such that $i_1+\ldots+i_s=n-v$ and $g^*_{\beta}(n-v)=e_{i_1}+\ldots+e_{i_s}$. Then, we construct $H$ as follows. Start from $H_0$ which is a copy of $F$ without a single edge. Clearly, $H_0$ is weakly $F$-saturated. Then, for every $j=1,\ldots,s$, assuming that a weakly $F$-saturated graph $H_{j-1}$ is constructed, set $V(H_j)=V(H_{j-1})\sqcup U_j$, where $U_j$ has size $i_j$, and $E(H_j)=E(H_{j-1})\sqcup E_j$, where $E_j$ has size $e_{i_j}$. The edges of $E_j$ are drawn in the following way: 
\begin{itemize}
\item let $\tilde F_j$ be a copy of $F$ with $U_j\subset V(\tilde F_j)$ such that $e_{i_j}+1$ is exactly the number of edges in $E(\tilde F_j)\setminus E(\tilde F_j\setminus U_j)$;
\item $E_j$ restricted on $U_j$ coincides with $E(\tilde F_j|_{U_j})$;
\item take an arbitrary set $W_j\subset V(H_{j-1})$ of size $|V(\tilde F_j)|-|U_j|=v-i_j$ and draw edges between $U_j$ and $W_j$ exactly in the same way as they appear between $U_j$ and $V(\tilde F_j)\setminus U_j$ in $\tilde F_j$;
\item remove a single edge from the final set of edges defined above.
\end{itemize}
It remains to prove that $H_j$ is weakly $F$-saturated, and then conclude that $H:=H_s$ is the desired weakly $F$-saturated graph on $n$ vertices. Note that there exists an $F$-bootstrap percolation process initiated at $H_j$ that ends up at the union $K$ of a clique on $V(H_{j-1})$ and $U_j\sqcup W_j$. Note that the intersection of these two sets $W_j$ has cardinality at least $\beta$ and both sets have cardinality at least $v$. 
 Thus, due to the definition of $\beta$, for every choice of $x\in V(H_{j-1})\setminus W_j$ and $y\in U_j$, a copy of $F$ can be placed inside $V(H_j)$ in such a way that $\{x,y\}$ is the only edge that does not belong to the union of cliques $K$. Thus, all the remaining edges of the clique on $V(H_j)$ can be activated, completing the proof.\\

It remains to prove that $\{i\in\mathcal{K}:\,\mathrm{wsat}(i+v,F)\leq e_i+(\ell-1)\}\subseteq\mathcal{K}_{\beta}$. Assume the contrary: take $i\in [v-\beta+1,v]\cap\mathcal{K}$ such that $\mathrm{wsat}(i+v,F)\leq e_i+(\ell-1)$. Let $H\in\underline{\mathrm{wSAT}}(v+i,F)$. Consider an $F$-bootstrap percolation process $H=H_0\subset\ldots\subset H_M=K_{v+i}$, each single edge from $H_j\setminus H_{j-1}$ creates $F_j\cong F$ in $H_j$. For every $j\in[M]$, let $U_j=V(F_1)\cup\ldots\cup V(F_j)$. Note that $U_M=V(H)$ since otherwise there exists a vertex $v\in V(H)$ adjacent to all the other vertices in $H$ and such that $H\setminus\{v\}$ is weakly $F$-saturated. But then deleting all but $\delta-1$ edges going from $v$ keeps the graph $H$ weakly $F$-saturated --- a contradiction with its minimality. Now, consider an inclusion-maximum sequence $U_1\subset U_{j_1}\subset\ldots\subset U_{j_h}$ such that each successive set is strictly bigger than its predecessor. Set $S_t=V(F_{j_t})\setminus U_{j_{t-1}}$, $t\in[h]$, where $j_0=1$. Note that 
$$
i_t:=|S_t|\leq v+i-|U_{j_{t-1}}|\leq v+i-|U_1|=i,
$$
that $e_{i_t}\leq E(H|_{U_{j_t}})\setminus E(H|_{U_{j_{t-1}}})$, and that $i_1+\ldots+i_h=i$. Then $\mathrm{wsat}(v+i,F)\geq\ell-1+\sum_{j=1}^h e_{i_j}$, and we get that 
$$
 e_i+(\ell-1)\geq\mathrm{wsat}(v+i,F)\geq\ell-1+\sum_{j=1}^h e_{i_j}.
$$
The obtained inequality $e_i\geq\sum_{j=1}^h e_{i_j}$ may only happen if $h=1$ and $i_1=i$ since $i\in\mathcal{K}$. But then $\mathrm{wsat}(v+i,F)=\ell-1+e_i$. It means that there is a bootstrap $F$-percolation process initiated at $H$ such that, firstly, a clique on $V_1:=V(F_1)$ is activated, secondly, a clique on $V_2:=V(F_{j_1})$ is activated, and, finally, all edges between $V_1\setminus V_2$ and $V_2\setminus V_1$ are activated. Let us take the first edge $\{x,y\}$ that appears between $V_1\setminus V_2$ and $V_2\setminus V_1$ and note that 
$$
|V_1\cap V_2|=|V_2|-|V_2\setminus V_1|=v-i_1=v-i\leq\beta-1.
$$
But then we get that there exists a copy of $F$ such that the deletion of at most $|V_1\cap V_2|\leq\beta-1$ vertices from it creates the cut-edge $\{x,y\}$ --- a contradiction with the definition of $\beta$.
\end{proof}

\section*{Appendix: the proof of Faudree, Gould and Jacobson}

Let us recall the proof of the lower bound~(\ref{wsat:lower_bound_general}) due to Faudree, Gould and Jacobson~\cite{Faudree}:

{\small
\begin{center}
Assume that $G\in\mathrm{wSAT}(n,F)$ for $n$ sufficiently large. Partition the vertices of $G$ into $A$ and $B$, with $B$ being the vertices of degree $\delta-1$ and $A$ the remaining vertices. Let $|B|=k$, and so $|A|=n-k$. The vertices in $B$ form an independent set, since the addition of a first edge to $v\in B$ must result in $v$ and all of its neighbors having degree at least $\delta$. Likewise, each vertex in $A$ must have degree at least $\delta - 2$ relatively to $A$ to be able to add edges between $B$ and $A$, since at most 2 vertices of $B$ can be used. This gives the inequality $(\delta-1)k+(\delta-2)(n-k)\leq\delta(n-k)$. This implies $k\leq 2n/(\delta+1)$, and so $\mathrm{wsat}(n,F)\geq\frac{\delta n}{2}-\frac{n}{\delta+1}.$
\end{center} 
}


First, the counting argument $(\delta-1)k+(\delta-2)(n-k)\leq\delta(n-k)$ is unclear and seems to be false since $k$ could be greater than $2n/(\delta+1)$. In particular, the set $B$ of $H_{v,n}\in\underline{\mathrm{wSAT}}(n,K_v)$ obtained by drawing all edges from vertices on $[n]\setminus[v-2]$ to the clique on $[v-2]$ has cardinality $n-\delta+1$ which is bigger than $2n/(\delta+1)$ for all $\delta>1$. Second, the lower bound for $\mathrm{wsat}(n,F)$ may follow only from a lower bound on $k$ (we have the bound $\mathrm{wsat}(n,F)\geq (\delta-1)k + \frac{(\delta-2)(n-k)}{2}=\frac{\delta-2}{2}n+k\frac{\delta}{2}$ that increases in $k$). So, as might appear at first sight, the authors wrote the opposite inequality --- it might be $k\geq 2n/(\delta+1)$. However, $B$ could be even empty (take $K_n\in\mathrm{wSAT}(n,K_v)$) and then $k=0$.

\end{document}